\documentclass[reqno]{amsart}

\usepackage{amssymb}
\usepackage{amsthm}
\usepackage{amsmath}
\usepackage{float}
\usepackage{graphicx}
\usepackage{enumitem}

\theoremstyle{plain}
 \newtheorem{theorem}{Theorem}[section]
 \newtheorem{lemma}[theorem]{Lemma}
 \newtheorem{corollary}[theorem]{Corollary}

\theoremstyle{definition}

 \newtheorem{remark}[theorem]{Remark}

\numberwithin{equation}{section}

\renewcommand{\a}{\alpha}
\renewcommand{\b}{\beta}
\newcommand{\g}{\gamma}
\newcommand{\G}{\Gamma}
\renewcommand{\d}{\delta}

\renewcommand{\k}{\kappa}

\renewcommand{\o}{\omega}

\newcommand{\ps}{\emptyset}
\newcommand{\sst}{\subseteq}

\newcommand{\rest}{\upharpoonright}

\newcommand{\conc}{{}^\smallfrown\!\!}

\newcommand{\ddd}{\mathfrak D}

\newcommand{\cc}{\mathcal C}
\newcommand{\cd}{\mathcal D}

\newcommand{\seq}[1]{\left<#1\right>}
\newcommand{\set}[1]{\left\{#1\right\}}
\newcommand{\abs}[1]{\left\vert#1\right\vert}

\newcommand{\otp}{\operatorname{otp}}
\newcommand{\cf}{\operatorname{cf}}

\author[B. Kuzeljevic]{Borisa Kuzeljevic}
\address{Department of mathematics and informatics, University of Novi Sad, Serbia}
\curraddr{}
\email{borisha@dmi.uns.ac.rs}
\urladdr{}
\dedicatory{}

\author[S. Todorcevic]{Stevo Todorcevic}
\address{Department of Mathematics, University of Toronto, Toronto, Canada, M5S 2E4. Institut de Math\'ematiques de Jussieu, UMR 7586, 2 pl. Jussieu, Case 7012, 75251 Paris Cedex 05, France. Mathematical Institute SANU, Kneza Mihaila 36, 11001 Belgrade, Serbia.}
\curraddr{}
\email{stevo@math.toronto.edu}
\email{stevo.todorcevic@imj-prg.fr}
\email{stevo@mi.sanu.ac.rs}
\urladdr{}
\dedicatory{}

\title[Cofinal types on $\omega_2$]{Cofinal types on $\omega_2$}
\subjclass[2010]{}
\thanks{The first author has been partially supported by the Science Fund of the Republic of Serbia grant no. 6062228}
\date{July 25, 2021}

\begin{document}

\begin{abstract}
In this paper we start the analysis of the class $\mathcal D_{\aleph_2}$, the class of cofinal types of directed sets of cofinality at most $\aleph_2$.
We compare elements of $\mathcal D_{\aleph_2}$ using the notion of Tukey reducibility.
We isolate some simple cofinal types in $\mathcal D_{\aleph_2}$, and then proceed to show which of these types have an immediate successor in the Tukey ordering of $\mathcal D_{\aleph_2}$.
\end{abstract}

\maketitle

\begin{figure}[H]

\unitlength 0.7mm
\linethickness{0.5pt}

\begin{picture}(100,125)(0,0)


\put(50,4){\line(2,1){40}}
\put(50,4){\line(-2,1){40}}
\put(50,4){\line(0,1){20}}
\put(10,24){\line(0,1){20}}
\put(10,24){\line(2,1){40}}
\put(90,24){\line(-2,1){40}}
\put(50,24){\line(-2,1){40}}
\put(10,44){\line(2,1){40}}
\put(90,44){\line(-2,1){40}}
\put(50,64){\line(2,1){40}}
\put(50,64){\line(-2,1){40}}
\put(50,24){\line(2,1){40}}
\put(90,24){\line(0,1){20}}
\put(10,44){\line(0,1){20}}
\put(50,44){\line(0,1){20}}
\put(90,44){\line(0,1){20}}
\put(10,64){\line(0,1){20}}
\put(90,64){\line(0,1){20}}
\put(10,84){\line(2,1){40}}
\put(90,84){\line(-2,1){40}}
\put(50,104){\line(0,1){20}}


\put(52,2){\makebox(0,0)[l]{$1$}}
\put(8,23){\makebox(0,0)[r]{$\o$}}
\put(92,23){\makebox(0,0)[l]{$\o_2$}}
\put(52,23){\makebox(0,0)[l]{$\o_1$}}
\put(8,45){\makebox(0,0)[r]{$\o\times\o_1$}}
\put(8,64){\makebox(0,0)[r]{$[\o_1]^{<\o}$}}
\put(8,85){\makebox(0,0)[r]{$\o_2\times [\o_1]^{<\o}$}}
\put(92,45){\makebox(0,0)[l]{$\o_1\times\o_2$}}
\put(92,64){\makebox(0,0)[l]{$[\o_2]^{\le\o}$}}
\put(52,45){\makebox(0,0)[l]{$\o\times\o_2$}}
\put(53,64){\makebox(0,0)[l]{$\o\times\o_1\times\o_2$}}
\put(54,105){\makebox(0,0)[l]{$[\o_1]^{<\o}\times [\o_2]^{\le\o}$}}
\put(92,85){\makebox(0,0)[l]{$\o\times [\o_2]^{\le\o}$}}
\put(52,125){\makebox(0,0)[l]{$[\o_2]^{<\o}$}}
    \put(50,124){\circle*{1.5}}
    \put(10,24){\circle*{1.5}}
    \put(50,24){\circle*{1.5}}
    \put(90,24){\circle*{1.5}}
    \put(10,44){\circle*{1.5}}
    \put(10,64){\circle*{1.5}}
    \put(10,84){\circle*{1.5}}
    \put(50,44){\circle*{1.5}}
    \put(50,64){\circle*{1.5}}
    \put(50,104){\circle*{1.5}}
    \put(90,84){\circle*{1.5}}
    \put(90,64){\circle*{1.5}}
    \put(90,44){\circle*{1.5}}
    \put(90,24){\circle*{1.5}}
    \put(50,4){\circle*{1.5}}

\end{picture}
\caption{Tukey ordering of simple elements of the class $\mathcal D_{\aleph_2}$}\label{slika}
\end{figure}
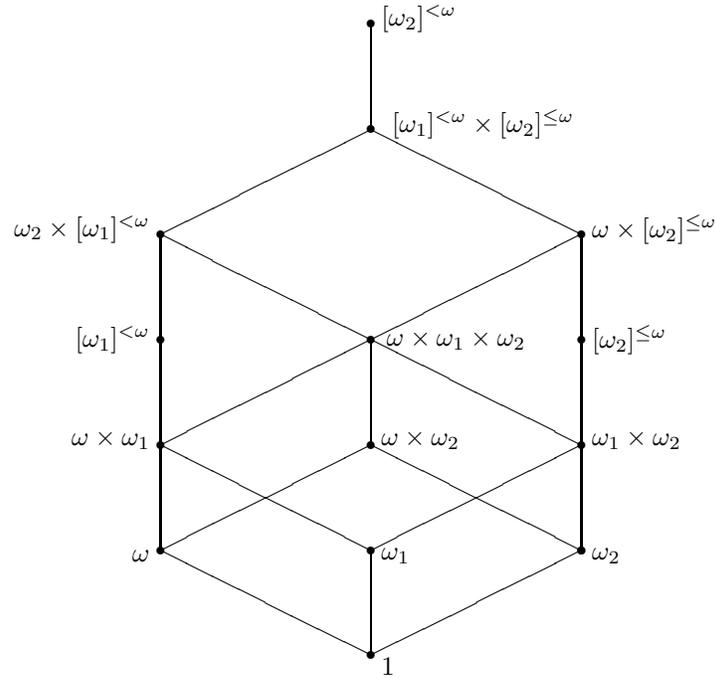

\section{Introduction}

The purpose of this paper is to start the analysis of the class $\mathcal D_{\aleph_2}$, the class of directed sets of cofinality at most $\aleph_2$.
Up to now, there is a satisfactory understanding of cofinal types of directed sets of cardinality at most $\aleph_1$.
This line of research originates from the work of Tukey in \cite{tukey}, and from the work of Birkhoff in \cite{birkhoff} and Day in \cite{mmday}.
Afterwards, Isbell in \cite{isbell1,isbell2} and Schmidt in \cite{schmidt}, continued with finer analysis of the class of cofinal types of all directed sets.
In the early 1980s, the second author in \cite{stevo} completed Isbell's investigation of cofinal types of directed sets of cardinality at most $\aleph_1$, by proving that only five cofinal types in $\mathcal D_{\aleph_1}$ can be found in ZFC without additional set-theoretic assumptions.
These are $1$, $\o$, $\o_1$, $\o\times\o_1$, and $[\o_1]^{<\o}$.
In the same paper, the second author proved that, assuming the Continuum Hypothesis, this class can be very rich.
In particular, under CH, the class $\mathcal D_{\aleph_1}$ contains $2^{\mathfrak c}$ pairwise cofinally non-equivalent directed sets.
After that, he was also able to extend those results to all transitive relations on $\o_1$ in \cite{transitive}.

It is worth mentioning, although it is not directly connected to our work in this paper, that there has been a significant amount of work on cofinal types of definable directed sets, and cofinal types of ultrafilters when viewed as directed sets.
The work on definable directed sets is due to the second author and Solecki in \cite{stevoslawek,stevoslawek2}.
The work on cofinal types of ultrafilters viewed as directed sets started with the work of Milovich in \cite{milovich}, and continued through papers of the second author with Dobrinen \cite{stevonatasha,stevonatasha1,stevonatasha2},  and with Raghavan in \cite{stevodilip}.
Most recently, Raghavan and Shelah in \cite{raghavanshelah}, and the first author and Raghavan in \cite{borisadilip} contributed to this topic.

The structure of the simplest directed sets in $\mathcal D_{\aleph_2}$ is presented in Figure \ref{slika}. 
Note that the order of a directed set in this picture is always assumed to be the standard one, $<$ for ordinals, $\sst$ for families of sets, and the product ordering for products of directed sets.
At the top is $[\o_2]^{<\o}$, the maximal cofinal type of directed sets of cardinality $\aleph_2$.
At the bottom is $1$, the minimal cofinal type of all directed sets.
All the other simple types are obtained as products of $\o$, $\o_1$, $\o_2$, $[\o_1]^{<\o}$, and $[\o_2]^{\le\o}$.
Note that the set $[\o_2]^{\le\o}$ is not of cardinality $\aleph_2$, but it contains a cofinal subset $\ddd$ of cardinality $\aleph_2$, as Lemma \ref{existence} shows.

The paper is organized as follows. In the second section we introduce all the relevant notions. In the third section we prove some basic inequalities between cofinal types of directed set in $\mathcal D_{\aleph_2}$. In this section we prove that all these inequalities are strict, i.e. these cofinal types are all different with resepect to cofinal equivalence. In the fourth section we prove that, under certain set-theoretic assumptions, there is a directed set between $\o_1\times\o_2$ and $\ddd$. Namely, we prove:

\begin{theorem}\label{prvaizmedju}
 Assume GCH and that there is a non-reflecting stationary subset of $S^2_0$.
 Then there is a directed set $D$ such that $\o_1\times\o_2<_T D<_T \ddd$
\end{theorem}

Finally, in the last section, we show that, under the same set-theoretic assumptions, there is a directed set between $[\o_1]^{<\o}\times \ddd$ and $[\o_2]^{<\o}$.

\begin{theorem}\label{drugaizmedju}
 Assume GCH and that there is a non-reflecting stationary subset of $S^2_0$.
 Then there is a directed set $D$ such that $[\o_1]^{<\o}\times \ddd<_T D<_T [\o_2]^{<\o}.$
\end{theorem}

\section{Preliminaries}

We use standard set theoretic notation.
In particular, if $A$ is a set and $\kappa$ is a cardinal, then $[A]^{\kappa}$ denotes the set of all subsets of $A$ of cardinality $\kappa$.
Thus, for example, $[A]^1=\set{\set{x}:x\in A}$, whereas $[A]^{\omega}$ is the set of all countably infinite subsets of the set $A$, and $[A]^{\le\omega}$ is the set of all at most countable subsets of $A$.
For a function $f:X\to Y$ and a set $A\sst X$, we denote $f''A=\set{f(x):x\in A}$.

A partially ordered set $\seq{X,<}$ is \emph{directed} if for any $x$ and $y$ in $X$ there is some $z$ in $X$ such that $x<z$ and $y<z$.
A directed partially ordered set is called just a \emph{directed set}.
A subset $Y$ of a directed set $X$ is \emph{bounded} if there is some $x$ in $X$ such that $y\le x$ for each $y$ in $Y$.
Otherwise, $Y$ is \emph{unbounded} in $X$.
If $D$ and $E$ are two directed sets, we say that $f:D\to E$ is a \emph{Tukey function} if $f''X$ is unbounded in $E$ whenever $X$ is unbounded in $D$.
When there is a Tukey function from a directed set $D$ into a directed set $E$, we say that $D$ is \emph{Tukey reducible to} $E$, and write $D\le_T E$.
If there is a Tukey function $f:D\to E$, but there is no Tukey function from $E$ to $D$, then we write $D<_T E$.
Note that $\le_T$ is a transitive relation.
For two directed sets $D$ and $E$, by results in \cite{schmidt} and \cite{tukey}, $D\le_T E$ if and only if there is a function $g:E\to D$ such that for every $d\in D$ there is an $e\in E$ such that $g(e')\ge d$ for each $e'\ge e$.
A function $g$ with this property is called \emph{a convergent map} from $E$ to $D$.
A subset $C$ of a directed set $D$ is said to be \emph{cofinal} in $D$ if for every $d$ in $D$ there is some $c$ in $C$ such that $d\le c$.
We say that two directed sets $D$ and $E$ are \emph{cofinally similar} if they are isomorphic to cofinal subsets of a single partially ordered set.
Recall that by results of Tukey in \cite{tukey}, two directed sets $D$ and $E$ are cofinally similar ($D\equiv_T E$) if and only if $D\le_T E$ and $E\le_T D$.
He also showed that $\equiv_T$ is an equivalence relation.
The equivalence classes of this relation are called \emph{cofinal types}.
For a directed set $D$, its \emph{cofinality} is the minimal cardinality of a cofinal subset of $D$.
Recall that $\cf(D)$ denotes the cofinality of a directed set $D$, and that $\cd_{\k}$ denotes the class of cofinal types of directed sets of cofinality at most $\k$.

Suppose that $\k$ is a cardinal. A set $C\sst \k$ is a club in $\k$ if it is closed and unbounded in $\k$. A set $S\sst \k$ is stationary in $\k$ if $S\cap C\neq \ps$ for every club $C$ in $\k$.
A set $S$ stationary in $\k$ is said to be \emph{non-reflecting} if for every $\g<\k$ of uncountable cofinality, the set $S\cap \g$ is not stationary in $\g$. Regarding notation, we will also be using $S^2_0=\set{\a<\o_2:\cf(\a)=\o}$ and $S^2_1=\set{\a<\o_2:\cf(\a)=\o_1}$.

For directed sets $D_i$ ($i\in I$), their product is the set $\prod_{i\in I}D_i$ equipped with a relation $\le$ defined as follows: $\seq{d_i:i\in I}\le \seq{e_i:i\in I}$ iff $d_i\le_{D_i} e_i$ for each $i\in I$.
If $j\in I$, then $\pi_{D_j}:\prod_{i\in I}D_i\to D_j$ denotes the projection to $D_j$, i.e. $\pi_{D_j}(d_i:i\in I)=d_j$.

\begin{lemma}\label{upperbound}
 Suppose that $D_1,\dots,D_n$ are directed sets.
 If $D_1,\dots,D_n\le_T D$, then $D_1\times\cdots\times D_n\le_T D$.
\end{lemma}

\begin{proof}
 By \cite[Proposition 2]{directed}, $D_1\times\cdots\times D_n$ is the least upper bound of $D_1,\dots,D_n$.
 Since $D_1,\dots,D_n\le_T D$, it must be that $D_1\times\cdots\times D_n\le_T D$.
\end{proof}

\begin{lemma}\label{strictgeneral}
 Suppose that $D$ and $E$ are Tukey incomparable cofinal types. Then there is no Tukey map from $D\times E$ to either $D$ or $E$.
\end{lemma}

\begin{proof}
 By Proposition 2 of \cite{directed}, both $D$ and $E$ are Tukey below $D\times E$. If there were a Tukey map $f:D\times E\to E$, then we would have $D\le_T D\times E\le_T E$, and consequently $D\le_T E$ which is in contradiction with the assumption of the lemma.
\end{proof}

\begin{lemma}\label{existence}
 Directed set $[\o_2]^{\le\o}$ contains a cofinal subset $\ddd$ of size $\aleph_2$ with the property that every uncountable subset of $\ddd$ is unbounded in $[\o_2]^{\le\o}$.
 In particular, $[\o_2]^{\le \o}$ belongs to $\cd_{\aleph_2}$, i.e. $\cf\left([\o_2]^{\le\o}\right)\le\aleph_2$.
\end{lemma}

\begin{proof}
 First, for each $\a<\o_2$ fix an injection $e_{\a}:\a+1\to \o_1$ such that $e_{\a}(\a)=0$.
 Now for every $\g<\o_1$ denote $$F_{\g}(\a)=\set{\xi\le\a:e_{\a}(\xi)<\g}.$$
 Note that for every $\o_1\le \a<\o_2$ there is an unbounded set $C_{\a}\sst \o_1$ such that $\g\sst F_{\g}(\a)$ whenever $\g\in C_{\a}$.
 To see this suppose it is not the case, i.e. that there is a $\d<\o_1$ such that for each $\g>\d$ we have $\g\nsubseteq F_{\g}(\a)$.
 This means that for every $\g>\d$ there is $\xi_{\g}<\g$ such that $e_{\a}(\xi_{\g})\ge\g$.
 Let $f:\o_1\setminus (\d+1)\to\o_1$ be given by $f(\g)=\xi_{\g}$.
 By the pressing down lemma, there is a stationary set $S\sst \o_1\setminus (\d+1)$ (stationary subset of $\o_1$), such that $\xi_{\g}=\xi$ for each $\g\in S$.
 But then $e_{\a}(\xi)>\g$ for every $\g\in S$ which is not possible because $S$ is unbounded in $\o_1$.
 Thus, we showed that there is an unbounded $C_{\a}\sst \o_1$ such that $\g\sst F_{\g}(\a)$ for every $\g\in C_{\a}$.
 Define now
 \[\ddd=\set{F_{\g}(\a):\o_1\le\a<\o_2\ \&\ \g\in C_{\a}}.\]
 
 We will prove that the set $\ddd$ is as required.
 Since $\a$ ranges over a subset of $\o_2$ and $\g$ ranges over a subset of $\o_1$, it is clear that $\ddd$ is of cardinality $\aleph_2$.
 Since every $e_{\a}$ is 1-1 function, and every $\g\in C_{\a}$ is a countable ordinal, it follows that each $F_{\g}(\a)$ is a countable set.
 Thus $\ddd\sst [\o_2]^{\le\omega}$.
 Next, we prove that $\ddd$ is cofinal in $[\o_2]^{\le\omega}$.
 Take any countable $B\sst \o_2$.
 Then there is some $\o_1\le \a<\o_2$ such that $B\sst \a$.
 Since $B$ is countable, $e_{\a}$ is 1-1, and $C_{\a}$ is unbounded in $\o_1$, there is some $\g\in C_{\a}$ such that $e_{\a}''B\sst \g$, i.e. $B\sst F_{\g}(\a)\in \ddd$.
 Thus, $\ddd$ is cofinal in $[\o_2]^{\le\omega}$.
 
 We still have to prove that every uncountable subset of $\ddd$ is unbounded in $[\o_2]^{\le\omega}$.
 Take any uncountable $X\sst \ddd$.
 Let us enumerate $X=\set{F_{\g_{\xi}}(\a_{\xi}): \xi<\o_1}$.
 We consider two cases: when the set $\Lambda=\set{\alpha_{\xi}: \xi<\o_1}$ is uncountable or when the set $\Gamma=\set{\g_{\xi}:\xi<\o_1}$ is uncountable.
 There is no other case possible because if both $\Lambda$ and $\Gamma$ were countable, then the set $X$ would also be countable.
 If $\Lambda$ is uncountable, then from definition of $e_{\a}$, in particular from $e_{\a}(\a)=0$, it follows that $\Lambda\sst \bigcup X$.
 Thus $\bigcup X$ is uncountable, so $X$ cannot be bounded in $[\o_2]^{\le\omega}$.
 Suppose now that $\Gamma$ is uncountable.
 Since each $\g_{\xi}\in \Gamma$ belongs to $C_{\a_{\xi}}$, we know that $\g_{\xi}\sst F_{\g_{\xi}}(\a_{\xi})$ for $\xi<\o_1$.
 Since $\Gamma$ is an uncountable set of ordinals, $\bigcup\Gamma$ is also uncountable.
 Now we have $\bigcup\Gamma=\bigcup_{\xi<\o_1}\g_{\xi}\sst \bigcup_{\xi<\o_1}F_{\g_{\xi}}(\a_{\xi})=\bigcup X$, so $\bigcup X$ is again uncountable.
 Consequently, $X$ cannot be bounded in $[\o_2]^{\le\omega}$.
 We showed that in both cases $X$ in unbounded, so we conclude that every uncountable subset of $\ddd$ is unbounded in $[\o_2]^{\le\omega}$
\end{proof}

For the remaining of this paper $\ddd$ will denote the directed set defined in the proof of Lemma \ref{existence}, thus $\ddd$ is a cofinal subset of $[\o_2]^{\le\o}$, the cardinality of $\ddd$ is $\aleph_2$, and every uncountable subset of $\ddd$ is unbounded. 

\begin{remark}\label{alternative}
 Note that since $[\o_2]^{\o}$ is cofinal in $[\o_2]^{\le\o}$ we have $$[\o_2]^{\omega}\equiv_T [\o_2]^{\le\omega}\equiv_T \ddd,$$ and so, depending on the situation, we will be using these three forms of the same cofinal type of a directed set $\ddd$.
\end{remark}

\section{Basic inequalities in $\mathcal D_{\aleph_2}$}

\begin{lemma}[see \cite{directed}]\label{poznate}
 There are Tukey maps:
 \begin{itemize}
  \item $f_0:1\to \o_1$;
  \item $f_1:1\to \o$;
  \item $f_2:\o\to \o\times\o_1$;
  \item $f_3:\o_1\to \o\times\o_1$;
  \item $f_4:\o\times\o_1\to [\o_1]^{<\o}$.
 \end{itemize}
\end{lemma}

\begin{lemma}\label{ocigledne}
 There are Tukey maps:
 \begin{itemize}
  \item $f_5:1\to \o_2$;
  \item $f_6:\o_2\to \o\times\o_2$;
  \item $f_7:\o_2\to \o_1\times\o_2$;
  \item $f_8:\o\time\o_1\to \o_1\times\o_2$;
  \item $f_9:[\o_1]^{<\o}\to \o_2\times [\o_1]^{<\o}$;
  \item $f_{10}:\o\to\o\times\o_2$;
  \item $f_{11}:\o\times\o_1\times\o_2\to \o_2\times [\o_1]^{<\o}$;
  \item $f_{12}:\o\times\o_2\to \o\times\o_1\times\o_2$;
  \item $f_{13}:\o\times\o_1\to \o\times \o_1\times\o_2$;
  \item $f_{14}:\o_1\times\o_2\to \o\times\o_1\times\o_2$;
  \item $f_{15}:[\o_1]^{<\o}\times [\o_2]^{\le \o}\to [\o_2]^{<\o}$.
 \end{itemize}
\end{lemma}

\begin{proof}
 To see that inequality (11) is true, first note that by \cite[Proposition 2]{directed}, $\o,\o_1,\o_2\le_T\o_2\times [\o_1]^{\le\o}$. Now Lemma \ref{upperbound} implies that $\o\times\o_1\times\o_2\le_T \o_2\times [\o_1]^{<\o}$, i.e. there is a required Tukey map.
 Inequalities (5) and (15) are obvious, and inequalities (6)-(10) and (12-14) follow directly from \cite[Proposition 2]{directed}.
\end{proof}

\begin{lemma}\label{belowcountablesubsetsofomega2}
 There is a Tukey map $f_{16}:\o_1\times\o_2\to [\o_2]^{\le\o}$.
\end{lemma}

\begin{proof}
 Fix a bijection $g:\o_1\times \o_2\to \o_2$, and let $f(\a,\b)=\set{g(\xi,\b):\xi<\a}$ for $\seq{\a,\b}\in\o_1\times\o_2$ .
 Note that $f(\a,\b)$ is a countable subset of $\o_2$ for each $\seq{\a,\b}\in\o_1\times\o_2$.
 To see that $f$ is a Tukey map, let $X$ be an unbounded subset of $\o_1\times\o_2$.
 Then, either $\pi_{\o_1}''X$ is unbounded in $\o_1$, or $\pi_{\o_2}''X$ is unbounded in $\o_2$.
 
 Suppose first that $\pi_{\o_1}''X$ is unbounded in $\o_1$, i.e. $\pi_{\o_1}''X$ is of cardinality $\aleph_1$.
 For each $\a\in \pi_{\o_1}''X$ choose $\b_{\a}$ such that $\seq{\a,\b_{\a}}\in X$.
 Since \[\textstyle\bigcup_{\a\in\pi_{\o_1}''X}\set{g(\xi,\b_{\a}):\xi<\a}\sst f''X,\] since $g$ is a bijection, and since $\pi_{\o_1}''X$ is of size $\aleph_1$, it must be that $\bigcup f''X$ is of cardinality at least $\aleph_1$, thus unbounded in $[\o_2]^{\le\o}$.
 
 Suppose now that $\pi_{\o_2}''X$ is unbounded in $\o_2$, i.e. $\pi_{\o_2}''X$ is of cardinality $\aleph_2$.
 Then there is $\a<\o_1$ such that $\Sigma=\set{\b<\o_2:\seq{\a,\b}\in X}$ is of cardinality $\aleph_2$, so since $g$ is bijection and
 \[
  \textstyle\bigcup_{\b\in \Sigma}\set{g(\xi,\b):\xi<\a}\sst f''X,
 \]
 it must be that $\bigcup f''X$ is of cardinality $\aleph_2$, thus unbounded in $[\o_2]^{\le\o}$.
\end{proof}

\begin{lemma}\label{top}
 There are Tukey maps:
 \begin{itemize}
  \item $f_{17}:[\o_2]^{\le\o}\to \o\times [\o_2]^{\le\o}$;
  \item $f_{18}:\o\times\o_1\times\o_2\to \o\times [\o_2]^{\le\o}$;
  \item $f_{19}:\o\times [\o_2]^{\le\o}\to [\o_1]^{<\o}\times [\o_2]^{\le\o}$;
  \item $f_{20}:\o_2\times [\o_1]^{<\o}\to [\o_1]^{<\o}\times [\o_2]^{\le\o}$.
 \end{itemize}
\end{lemma}

\begin{proof}
 Inequality (17) follows from \cite[Proposition 2]{directed}.
 Inequality (18) follows Lemma \ref{upperbound} applied to inequality (16).
 Inequality (19) follows Lemma \ref{upperbound} applied to inequalities (2) and (4).
 Inequality (20) follows Lemma \ref{upperbound} applied to inequalities (7) and (16).
\end{proof}

\begin{lemma}\label{strictspecific}
 There is no Tukey map from $[\o_2]^{\le \o}$ to $\o_1\times\o_2$.
\end{lemma}

\begin{proof}
 Suppose that $f:[\o_2]^{\le\o}\to \o_1\times\o_2$ is a Tukey map. 
 
 If $\abs{\pi_{\o_2}(f''[\o_2]^1)}\le \aleph_1$, then there is a set $Y\sst[\o_2]^1$ of cardinality $\aleph_2$, and an ordinal $\g<\o_2$ such that $\pi_{\o_2}(f(x))=\g$ for each $x$ in $Y$.
 Since $Y$ is of cardinality $\aleph_2$, there is a set $Z\sst Y$ of cardinality $\aleph_2$, and an ordinal $\b<\o_1$ such that $f(x)=\seq{\b,\g}$ for each $x$ in $Z$.
 So $f^{-1}(\b,\g)$ is unbounded in $[\o_2]^{\le\o}$, contradicting the fact that $f$ is a Tukey map.
 
 If $\abs{\pi_{\o_2}(f''[\o_2]^1)}=\aleph_2$, then there is a set $Y\sst \o_2$ of cardinality $\aleph_2$, and an ordinal $\b<\o_1$ such that $\seq{\b,\g}\in f''[\o_2]^1$ for each $\g$ in $Y$.
 Let $W$ be any subset of $\set{\b}\times Y$ of cardinality $\aleph_1$.
 The set $W$ is bounded in $\o_1\times\o_2$.
 Since $W$ is uncountable, the set $f^{-1}W$ is unbounded in $[\o_2]^{\le\o}$, again contradicting the assumption that $f$ is a Tukey map.
\end{proof}

\begin{lemma}\label{incomparableright}
 $\o\times\o_1\times\o_2$ and $[\o_2]^{\le\o}$ are incomparable cofinal types.
\end{lemma}

\begin{proof}
 Since $\o\times\set{\o}\times\set{\o}$ is unbounded in $\o\times\o_1\times\o_2$, and every countable set in $[\o_2]^{\le\o}$ is bounded, there is no Tukey function from $\o\times \o_1\times\o_2$ into $[\o_2]^{\le\o}$.
 Thus $\o\times \o_1\times\o_2\nleq_T [\o_2]^{\le\o}$.
 So suppose that there is a Tukey map $f:[\o_2]^{\le\o}\to \o\times\o_1\times\o_2$.
 
 If $\abs{f''[\o_2]^1}\le\aleph_1$, then there is a set $Y\sst [\o_2]^1$ of cardinality $\aleph_2$, and a triple $\seq{\b,\g,\d}$ such that $f(x)=\seq{\b,\g,\d}$ for each $x$ in $Y$.
 Thus $f^{-1}(\b,\g\,\d)$ is unbounded in $[\o_2]^{\le\o}$, contradicting the assumption that $f$ is a Tukey function.
 
 If $\abs{f''[\o_2]^1}=\aleph_2$, then there is a set $Y\sst \o_2$ of cardinality $\aleph_2$, an integer $n<\o$, and an ordinal $\b<\o_1$ such that the set $Z=\set{n}\times\set{\b}\times Y$ is contained as a subset in $f''[\o_2]^1$.
 Let $W$ be any subset of $Z$ of cardinality $\aleph_1$.
 Thus $W$ is bounded in $\o\times\o_1\times\o_2$.
 Since $W$ is uncountable  $f^{-1}W$ is unbounded in $[\o_2]^{\le\o}$, again contradicting the assumption that $f$ is a Tukey function.
\end{proof}

\begin{lemma}\label{incomparableleft}
 $\o\times\o_1\times\o_2$ and $[\o_1]^{<\o}$ are incomparable cofinal types.
\end{lemma}

\begin{proof}
 Suppose first that there is a Tukey function $f$ from $[\o_1]^{<\o}$ into $\o\times\o_1\times\o_2$.
 For every unbounded set $X\sst [\o_1]^{<\o}$, $\pi_{\o_2}''\left(f''X\right)$ is bounded in $\o_2$, thus $\pi_{\o\times\o_1}''\left(f''X\right)$ is unbounded in $\o\times\o_1$.
 Hence $\pi_{\o\times\o_1}\circ f$ is a Tukey function from $[\o_1]^{<\o}$ into $\o\times\o_1$, which is impossible.
 
 Suppose now that $f:\o\times\o_1\times\o_2\to [\o_1]^{<\o}$ is a Tukey function.
 Since $[\o_1]^{<\o}$ is of size $\aleph_1$, there is a set $Y\sst \o_2$ of cardinality $\aleph_2$, an integer $n<\o$, an ordinal $\b<\o_1$, and a finite set $F\sst \o_1$ such that $f(n,\b,\g)=F$ for each $\g$ in $Y$.
 Since $F$ is bounded in $[\o_1]^{<\o}$ and $\set{n}\times\set{\b}\times Y$ is unbounded in $\o\times\o_1\times\o_2$, this contradicts the assumption that $f$ is a Tukey function.
\end{proof}

\begin{lemma}\label{topleft}
 There is no Tukey function from $[\o_2]^{\le\o}$ into $\o_2\times [\o_1]^{<\o}$. 
\end{lemma}

\begin{proof}
 Assume that $f:[\o_2]^{\le\o}\to\o_2\times [\o_1]^{<\o}$ is a Tukey function.
 
 If $\abs{\pi_{\o_2}''(f''[\o_2]^1)}\le\aleph_1$, then there is a set $Y\sst [\o_2]^1$ of cardinality $\aleph_2$, and an ordinal $\b<\o_2$ such that $\pi_{\o_2}(f(x))=\b$ for each $x$ in $Y$.
 Since $[\o_1]^{<\o}$ is of size $\aleph_1$, there is a set $Z\sst Y$ of size $\aleph_2$, and a finite set $F\sst \o_1$ such that $f(x)=\seq{\b,F}$ for each $x$ in $Z$.
 Thus $f^{-1}(\b,F)$ is unbounded in $[\o_2]^{\le\o}$, contradicting the assumption that $f$ is a Tukey map.
 
 If $\abs{\pi_{\o_2}''(f''[\o_2]^1)}=\aleph_2$, then since $[\o_1]^{<\o}$ is of size $\aleph_1$, there is a set $Y\sst \o_2$ of size $\aleph_2$, and a finite set $F\subseteq \o_1$ such that the set $Z=Y\times \set{F}$ is a subset of $f''[\o_2]^1$.
 Let $W$ be any subset of $Z$ of size $\aleph_1$.
 Then $W$ is bounded in $\o_2\times [\o_1]^{<\o}$.
 Since $W$ is uncountable, $g^{-1}W$ is unbounded in $[\o_2]^{\le\o}$, again contradicting the assumption that $f$ is a Tukey map. 
\end{proof}

\begin{lemma}\label{topright}
 There is no Tukey function from $[\o_1]^{<\o}$ into $\o\times [\o_2]^{\le\o}$. 
\end{lemma}

\begin{proof}
 Suppose that there is a Tukey map $f:[\o_1]^{<\o}\to \o\times [\o_2]^{\le\o}$.
 If the set $f''([\o_1]^{<\o})$ is countable, then $\pi_{[\o_2]^{\le\o}}''\left(f''([\o_1]^{<\o})\right)$ is bounded in $[\o_2]^{\le\o}$.
 Thus for every unbounded $X\sst [\o_1]^{<\o}$, the set $\pi_{\o}''(f''X)$ is unbounded in $\o$.
 Then $\pi_{\o}\circ f$ would be a Tukey function from $[\o_1]^{<\o}$ into $\o$, which is impossible.
 Hence, the set $f''([\o_1]^{<\o})$ is uncountable.
 Then there is $n<\o$, and an infinite set $X\sst [\o_2]^{\le\o}$ such that $W=\set{n}\times X\sst f''([\o_1]^{<\o})$.
 Clearly, $W$ is bounded in $\o\times [\o_2]^{\le\o}$.
 Since $W$ is infinite, the set $f^{-1}W$ is unbounded in $[\o_1]^{<\o}$.
 This contradicts the assumption that $f$ is a Tukey function. 
\end{proof}

\begin{lemma}\label{topleftright}
 $\o_2\times [\o_1]^{<\o}$ and $\o\times [\o_2]^{\le\o}$ are Tukey incomparable cofinal types.
\end{lemma}

\begin{proof}
 First notice that by Proposition 2 of \cite{directed}, if $\o_2\times [\o_1]^{<\o}\le_T \o\times [\o_2]^{\le\o}$, then $[\o_1]^{<\o}\le_T \o\times [\o_2]^{\le\o}$ which is impossible by Lemma \ref{topright}.
 On the other hand, if $\o\times [\o_2]^{\le\o}\le_T \o_2\times [\o_1]^{<\o}$, then $[\o_2]^{\le\o}\le_T \o_2\times [\o_1]^{<\o}$ contradicting Lemma \ref{topleft}.
\end{proof}

\begin{lemma}\label{reverse}
 All the Tukey inequalities (0)-(20) are strict. 
\end{lemma}

\begin{proof}
 First note that $\o,\o_1,\o_2$ are incomparable cofinal types, and that they are strictly above $1$, i.e. inequalities (0), (1), and (5) are strict.
 Lemma \ref{strictgeneral} implies that inequalities (2),(3),(6),(7),(8), and (10) are strict.
 For the same reason, and the fact that $\o\times\o_1,\o\times\o_2,\o_1\times\o_2$ are obviously incomparable, inequalities (12),(13), and (14) are strict.
 Lemma \ref{incomparableleft} implies that inequalities (9) and (11) are strict, and strictness of the inequality (4) is known since the introduction of the notion of cofinal equivalence (see \cite{directed}).
 By Lemma \ref{strictspecific} the inequality (16) is strict.
 Next, Lemma \ref{incomparableright} implies that the inequalities (17) and (18) are strict.
 Similarly, that the inequalities (19) and (20) are strict follows from Lemma \ref{topleftright}.
 
 We still have to prove that the inequality (15) is strict.
 Suppose that $f$ is a Tukey function  from $[\o_2]^{<\o}$ into $[\o_2]^{\le\o}\times [\o_1]^{<\o}$.
 
 If $f''[\o_2]^1$ is of size less then $\aleph_2$, then there is a set $Y\subseteq [\o_2]^1$ of cardinality $\aleph_2$, a countable set $S\sst \o_2$, and a finite set $F\sst \o_1$ such that $f(x)=\seq{S,F}$ for each $x$ in $[\o_2]^1$.
 Thus $f^{-1}(S,F)$ is unbounded in $[\o_2]^{<\o}$ contradicting the assumption that $f$ is a Tukey function.
 
 If $f''[\o_2]^1$ is of cardinality $\aleph_2$, then there is a finite $F\sst \o_1$ and an infinite set $Y\sst [\o_2]^{\le\o}$ such that $W=Y\times\set{F}\sst f''[\o_2]^1$.
 Then $W$ is bounded in $[\o_2]^{\le\o}\times [\o_1]^{<\o}$ while $f^{-1}W$ is unbounded in $[\o_2]^{<\o}$, contradicting the assumption that $f$ is a Tukey map.
\end{proof}




\section{Gaps in $\mathcal D_{\aleph_2}$}

\begin{theorem}
	There is no directed set $D$ such that $[\o_1]^{<\o} <_T D <_T \o_2\times [\o_1]^{<\o}$.
\end{theorem}

\begin{proof}
	Note that it is enough to prove that if $D$ is a directed set such that $[\o_1]^{<\o}<_T D \le_T \o_2\times [\o_1]^{<\o}$, then $D\equiv_T \o_2\times [\o_1]^{<\o}$.
	So suppose that $D$ be a directed set such that there are Tukey maps $f:[\o_1]^{<\o}\to D$ and $g:D\to \o_2\times [\o_1]^{<\o}$.
	Since $[\o_1]^{<\o}<_T D$, it must be that $D$ is of cofinality $\aleph_2$.
	For $s\in D$ let us denote $s^\uparrow=\set{x\in D: s\le x}$.
	Denote also $E=f''([\o_1]^{<\o})$.
	Since $f$ is a Tukey map, every infinite subset of $E$ is unbounded in $D$.
	Enumerate $[\o_1]^{<\o}=\set{s_{\xi}:\xi<\o_1}$.
	For $\a<\o_1$ consider sets $D_{\a}=g^{-1}[\o_2\times \set{s_{\a}}]$.
	Clearly $D=\bigcup_{\a<\o_1}D_{\a}$.
	Since $g$ is a Tukey map, for each $\a<\o_1$ if $X\sst D_{\a}$ is of cardinality $\aleph_1$, then $X$ is bounded in $D$.
	For each $\xi<\o_1$ pick $\a(\xi)$ such that $\a(\xi)>\a(\nu)$ for all $\nu<\xi$, and that $f(s_{\xi})^{\uparrow}\cap D_{\a(\xi)}$ is of cardinality $\aleph_2$.
	Let us enumerate $f(s_{\xi})^{\uparrow}\cap D_{\a(\xi)}=\set{d_{\xi\g}: \g<\o_2}$ for each $\xi<\o_1$.
	Now define $h:\o_2\times [\o_1]^{<\o}\to D$ so that $h(\g,s_{\xi})=d_{\xi\g}$.
	Clearly $h$ is well-defined, and we proceed to show that $h$ is a Tukey map.
	
	Suppose that $A\subset \o_2\times [\o_1]^{<\o}$ is unbounded.
	Then either there is an unbounded set $\Gamma\sst \o_2$ and $\xi<\o_1$ such that $\seq{\g,\xi}\in A$ whenever $\g\in \Gamma$, or there is an infinite set $\Sigma\subset\o_1$ such that for every $\xi\in\Sigma$ there is $\g\in \o_2$ such that $\seq{\g,s_{\xi}}\in A$.
	Assume first that there is such unbounded $\Gamma\sst \o_2$, i.e. $\Gamma$ is of cardinality $\aleph_2$ and $\Gamma\times\set{s_{\xi}}\sst A$ for a fixed $\xi<\o_1$.
	Then $h''A\cap D_{\a(\xi)}$ is of size $\aleph_2$, thus it is unbounded in $D$.
	Assume now that there is such infinite $\Sigma$, i.e. that $\Sigma\sst \pi''_{[\o_1]^{<\o}}A$.
	Then the set $\set{\xi<\o_1: h''A\cap D_{\a(\xi)}\cap f(s_{\xi})^{\uparrow}\neq\ps}$ is infinite, so if it were bounded by some $d\in D$, there would be an infinite subset of $E$ bounded by $d$ which is impossible.
	Thus $h''A$ is unbounded in this case also.
\end{proof}

\begin{theorem}
	There is no directed set $D$ such that $[\o_2]^{\le\o}<_T D<_T \o\times [\o_2]^{\le\o}$.
\end{theorem}

\begin{proof}
	Suppose that $D$ is a directed set such that $[\o_2]^{\le\o}<_T D$, and that $D\le_T \o\times [\o_2]^{\le\o}$.
	Let, by Lemma \ref{existence}, $\ddd=\set{s_{\xi}:\xi<\o_2}$ be a cofinal subset of $[\o_2]^{\le\o}$ of size $\aleph_2$ such that every uncountable subset of $\ddd$ is unbounded.
	Note that it is enough to prove that $\o\times \ddd\le_T D$.
	Let $f$ be a Tukey function from $[\o_2]^{\le\o}$ to $D$.
	Note that then $E=f''\ddd$ is such that every uncountable subset of $E$ is unbounded in $D$.
	On the other hand, let $g$ be a Tukey function from $D$ to $\o\times [\o_2]^{\le\o}$.
	Denote $D_n=g^{-1}(\set{n}\times [\o_2]^{\le\o})$ for each $n<\o$.
	Then $D=\bigcup_{n<\o}D_n$.
	Now, for $k<\o$ define $$D_k'=D_k\setminus\set{d\in D: (\exists n<k)(\exists x\in D_n)\ d\le x}.$$
	Clearly, $\bigcup_{n<\o}D_n'$ is a partition of $D$, and $D_n'\sst D_n$ for each $n<\o$.
	Since $g$ is a Tukey function, for each $n<\o$ every countable subset of $D_n$ is bounded in $D$.
	Thus, for each $n<\o$ every countable subset of $D_n'$ is bounded in $D$.
	Since $E$ is of cardinality $\aleph_2$, there is $n_0$ such that $\abs{E\cap D_{n_0}'}=\aleph_2$.
	Denote $X_0=E\cap D_{n_0}'=\set{d_{\xi 0}:\xi<\o_2}$.
	For each $\xi<\o_2$ pick $d_{\xi 0}'\in D\setminus D_{n_0}'$ such that $d_{\xi 0}<d_{\xi 0}'$.
	Denote $X_0'=\set{d_{\xi 0}':\xi<\o_2}$.
	Now let $n_1$ be such that $\abs{X_0'\cap D_{n_1}'}=\aleph_2$.
	Such an $n_1$ exists because for no $k<\o$ is $\bigcup_{n<k}D_n'$ cofinal in $D$.
	Clearly $n_1>n_0$.
	Denote $X_1=X_1'\cap D_{n_1}'$ and enumerate $X_1=\set{d_{\xi 1}:\xi<\o_2}$.
	Now suppose that $n_0<\cdots<n_m$, and sets $X_0,\dots,X_m,X_0',\dots,X_m'$, and $D_{n_0}',\dots,D_{n_m}'$ have been defined.
	Let $X_m=\set{d_{\xi m}:\xi<\o_2}$.
	For each $\xi<\o_2$ pick $d_{\xi m}'\in D\setminus\bigcup_{n\le m}D_n'$ such that $d_{\xi m}<d_{\xi m}'$, and denote $X_m'=\set{d_{\xi m}':\xi<\o_2}$.
	As before, there is $n_{m+1}>n_m$ such that $\abs{X_m'\cap D_{n_{m+1}}'}=\aleph_2$.
	Denote $X_{m+1}=X_m'\cap D_{n_{m+1}}'=\set{d_{\xi m+1}:\xi<\o_2}$.
	Finally, define $h:\o\times \ddd\to D$ as follows: for $n<\o$ and $\xi<\o_2$ let $h(n,s_{\xi})=d_{n_{\xi}}$.
	
	We will prove that $h$ is a Tukey finction which will conclude the proof.
	Suppose that $Y\sst \o\times \ddd$ is unbounded.
	Then either $\pi_{\o}''Y$ is unbounded in $\o$ or $\pi_{\ddd}''Y$ is unbounded in $\ddd$.
	If $\pi_{\o}''Y$ is unbounded in $\o$, then the set $K=\set{n<\o: h''Y\cap D_n'\neq\ps}$ is infinite.
	If there were $d\in D$ such that $x\le d$ for all $x\in h''Y$, then there would be $m<\o$ such that $d\in D_m'$.
	But this is impossible because by the definition of $D_m'$ there is no $n>m$ and $x\in D_n'$ such that $x\le d$, and consequently the set $K$ cannot be infinite.
	Thus in this case $h''Y$ is unbounded.
	If $\pi_{\ddd}''Y$ is unbounded in $\ddd$, then for some $m<\o$ the set $h''Y\cap D_m'$ is uncountable.
	Suppose that $d\in D$ is such that $x\le d$ for each $x\in h''Y\cap D_m'$.
	But this is impossible because then there would be an uncountable subset of $E$ bunded by $d$.
	Thus $h''Y$ is unbounded in this case also, and so $h$ is a Tukey function from $\o\times \ddd$ to $D$.
\end{proof}

\section{Directed set between $\o_1\times\o_2$ and $[\o_2]^{\le\o}$}

In this section we prove that, under certain set theoretic assumptions, there is a directed set $D$ such that $\o_1\times\o_2<_T D<_T [\o_2]^{\le\o}$.
We will construct such a $D$ using an $\aleph_2$-Souslin tree.
Assume, in this entire section, that GCH holds and that there is a non-reflecting stationary subset of $S^2_0$.
By the work of Gregory in \cite{gregory}, this implies that there is an $\aleph_2$-Souslin tree.
So suppose that $T\sst \o_1^{<\o_2}$ is an $\aleph_2$-Souslin tree with the property that for each $t\in T$ and $\a<\o_1$, the node $t\conc\seq{\a}$ belongs to $T$.
Let $\le$ denote the order in the tree $T$, and let $T_{\a}$ denote the $\a$th level of the tree $T$.
For $X\sst T$, let $\widehat{X}$ be the downward closure in $T$ of the set $X$.
Finally, let $D_T$ be the set of all subsets $X$ of $T$ such that for each $t\in T$, the set $\set{\a<\o_1: t\conc\seq{\a}\in \widehat{X}}$ is non-stationary in $\o_1$.
We consider $D_T$ ordered by inclusion.
Since non-stationary sets in $\o_1$ form an ideal, $D_T$ is a directed set.

\begin{lemma}\label{denseinacone}
 Suppose that $X\sst T$ is of cardinality $\aleph_2$.
 Then there is an element $t$ of $T$ such that for each $t'\ge t$ there is $x\in X$ such that $x\ge t'$. 
\end{lemma}

\begin{proof}
 Suppose the contrary: that for each $t\in T$, there is $t'\ge t$ such that every $x\ge t'$ does not belong to $X$.
 By induction we construct an antichain $\seq{t_{\a}:\a<\o_2}$ in $T$.
 By applying the assumption on the root $\ps$ we obtain $t_0$ such that every $x\ge t_0$ does not belong to $X$.
 Suppose now that we are in stage $\b<\o_2$, and that $t_{\a}$ has been chosen for each $\a<\b$ in such a way that $\seq{t_{\a}:\a<\b}$ is an antichain, and that for each $\a<\b$, every $x\ge t_{\a}$ does not belong to $X$.
 Let $\d$ be the supremum of levels of the nodes $t_{\a}$ ($\a<\b$).
 Then $T\rest\d$ is of cardinality $\aleph_1$.
 Thus, there is a $t\in X$ on a level above $\d$, and such that $\widehat{\set{t}}\cap \set{t_{\a}:\a<\b}=\ps$.
 By the assumption, there is $t_{\b}\ge t$ such that every $x\ge t_{\b}$ does not belong to $X$.
 Clearly, $\seq{t_{\a}:\a\le\b}$ is an antichain in $T$.
 This completes the definition of an antichain of size $\aleph_2$ in $T$, which is in contradiction with the assumption that $T$ is an $\aleph_2$-Souslin tree.
\end{proof}

\begin{corollary}\label{cardbound}
 Every element $X$ of $D_T$ is of cardinality at most $\aleph_1$.
 In particular, $D_T$ is of cardinality $\aleph_2$.
\end{corollary}

\begin{proof}
 Let $X\in D_T$, and suppose that $X$ is of cardinality $\aleph_2$.
 By Lemma \ref{denseinacone} there is a $t$ in $T$ such that for each $t'\ge t$ there is $x$ in $X$ such that $x\ge t'$.
 This means that $\set{\a<\o_1:t\conc\seq{\a}\in\widehat{X}}=\o_1$, in particular this set is non-stationary, contradicting the fact that $X$ is in $D_T$.
 Thus $X$ must be of cardinality less than $\aleph_2$.
 Now, the standing assumption of GCH in this section implies that $\abs{D_T}=\aleph_2$.
\end{proof}

\begin{corollary}\label{part}
 For every partition $D_T=\bigcup_{\g<\o_1}D_{\g}$, there is an ordinal $\g<\o_1$, and an unbounded $E\sst D_{\g}$ of size $\aleph_1$.
\end{corollary}

\begin{proof}
 Recall that GCH is assumed in this section.
 Thus $\abs{D_T}=\aleph_2$.
 So there is a $\g<\o_1$ such that $\abs{D_{\g}}=\aleph_2$.
 Then $\abs{\bigcup D_{\g}}=\aleph_2$.
 By Lemma \ref{denseinacone}, there is a $t$ in $T$ such that for each $t'\ge t$ there is $x$ in $\bigcup D_{\g}$ such that $x\ge t'$.
 Again, this implies that $\set{\a<\o_1: t\conc\seq{\a}\in\widehat{\bigcup D_{\g}}}=\o_1$.
 Now for every $\a<\o_1$ pick $X_{\a}$ in $D_{\g}$ such that $t\conc\seq{\a}\in\widehat{X_{\a}}$.
 Let $E=\set{X_{\a}:\a<\o_1}$.
 Since any upper bound $Y$ for $E$ would have to contain the union of $E$ as a subset, we would have $\set{\a<\o_1:t\conc\seq{\a}\in\widehat{Y}}=\o_1$, in particular this set would be non-stationary.
 Thus $E$ is unbounded in $D_T$.
 Clearly $E$ is of cardinality $\aleph_1$ and a subset of $D_{\g}$.
\end{proof}

\begin{lemma}\label{fourth}
 If $D'$ is a cofinal subset of $D_T$, then there is an uncountable subset of $D'$ bounded in $D_T$.
\end{lemma}

\begin{proof}
 Let $D'$ be a cofinal subset of $D_T$, and let $\set{t_{\g}:\g<\o_2}$ be a level set in $T$.
 Note that $L\sst T$ is a level set if $\abs{L\cap T_{\g}}=1$ for each $\g<\o_2$.
 Clearly, for every $\g<\o_2$, $\set{t_{\g}}\in D_T$.
 Thus, for each $\g<\o_2$ we can choose $S_{\g}\in D'$ such that $t_{\g}\in S_{\g}$.
 Let, for each $X\in D_T$ and $t\in T$, $N^X_t$ denote the non-stationary set $\set{\a<\o_1:t\conc \seq{\a}\in\widehat{X}}$, and let $C^X_t$ denote the club in $\o_1$ disjoint with $N^X_t$.
 
 Now fix a sufficiently large regular cardinal $\theta$, and an elementary submodel $M\prec H(\theta)$ of cardinality $\aleph_1$ containing all the relevant objects and such that $M^{\aleph_0}\sst M$.
 Denote $\d=M\cap \o_2$.
 Then $\d$ is a limit ordinal which belongs to $\o_2$, so let $\seq{\g_{\xi}:\xi<\o_1}$ be an increasing sequence in $\d$ such that $\sup_{\xi<\o_1}\g_{\xi}=\d$.
 Enumerate $T\rest\d=\set{s_{\xi}: \xi<\o_1}$ in such a way that if $s_{\xi_1}\le s_{\xi_2}$, then $\xi_1\le \xi_2$.
 In order to simplify notation, let $C^{\g}_{\xi}$ denote the set $C^{S_{\g}}_{s_{\xi}}$ for each $\g<\o_2$ and $\xi<\o_1$.
 
 Now, by induction, we construct three sequences $\seq{\d_{\xi}:\xi<\o_1}$, $\seq{\G_{\xi}:\xi<\o_1}$, and $\seq{\eta_{\xi}:\xi<\o_1}$ with the following properties:
 \begin{enumerate}
  \item\label{treci} $C^{\d}_{\xi_1}\cap\eta_{\xi_2}=C^{\d_{\xi_2}}_{\xi_1}\cap \eta_{\xi_2}$ for $\xi_1\le\xi_2<\o_1$,
  \item\label{drugi} $\d_{\xi_1}<\d_{\xi_2}<\d$ for $\xi_1<\xi_2<\o_1$,
  \item\label{poslednji} $\sup\set{\d_{\xi}:\xi<\o_1}=\d$.
 \end{enumerate}

 First consider $s_0$.
 Let $\eta_0$ be the first limit point of $C^{\d}_{0}$.
 Let
 \[
  \G_0=\set{\g<\o_2: C^{\d}_{0}\cap \eta_0=C^{\g}_{0}\cap \eta_0}.
 \]
 Since $\d\in \G_0$, the set $\G_0$ is stationary in $\o_2$.
 Let $\d_0=\min\G_0$.
 
 Suppose now that $\xi_0<\o_1$, and that $\d_{\xi},\G_{\xi}$, and $\eta_{\xi}$ have been constructed for each $\xi<\xi_0$.
 Let $\eta_{\xi_0}$ be the first limit point of $C^{\d}_{\xi_0}\setminus\sup\set{\eta_{\xi}:\xi<\xi_0}$.
 Consider the set
 \[
  \G_{\xi_0}=\set{\g\in \bigcap_{\xi<\xi_0}\G_{\xi}:(\forall \xi\le\xi_0)\ C^{\d}_{\xi}\cap\eta_{\xi_0}=C^{\g}_{\xi}\cap\eta_{\xi_0}}.
 \]
 Since $\G_{\xi_0}$ belongs to $M$, and since $\d\in \G_{\xi_0}$, it must be that $\G_{\xi_0}$ is stationary in $\o_2$.
 Since $\G_{\xi_0}$ is cofinal in $\o_2$ and belongs to $M$, the set $\d\cap \G_{\xi_0}$ is cofinal in $\d$.
 Define $\d_{\xi_0}$ to be the minimal ordinal in $\d\cap \G_{\xi_0}$ greater then both $\sup_{\xi<\xi_0}\d_{\xi}$ and $\g_{\xi_0}$ (recall that $\seq{\g_{\xi}:\xi<\o_1}$ is chosen to be cofinal in $\d$).
 It is clear from the construction that conditions (\ref{treci}-\ref{poslednji}) are satisfied. 
 
 Now we prove that $\set{S_{\d_{\xi}}:\xi<\o_1}$ is as required.
 It is sufficient to prove that $S=\bigcup_{\xi<\o_1}S_{\d_{\xi}}\in D_T$, as it will witness that $\set{S_{\d_{\xi}}:\xi<\o_1}$ is an uncountable subset of $D'$ bounded in $D_T$.
 Since, for each $\xi<\o_1$, both $\d_{\xi}$ and $\seq{S_{\g}:\g<\o_2}$ belong to $M$, it must be that $S_{\d_{\xi}}\in M$.
 Since $\abs{S_{\d_{\xi}}}\le \aleph_1$ ($\xi<\o_1$), we have $S\sst M$.
 Thus $S\sst T\rest\d$.
 This means that, in order to prove $S\in D_T$, it is enough to prove that for each $t\in T\rest\d$, the set $\set{\a<\o_1:t\conc\seq{\a}\in\widehat{S}}$ is non-stationary in $\o_1$.
 So take any $t\in T\rest\d$.
 Let $\zeta<\o_1$ be such that $s_{\zeta}=t$.
 Define
 \[\textstyle
  C=C^{\d}_{\zeta}\cap\left(\bigcap_{\xi\le\zeta}C^{\d_{\xi}}_{\zeta}\right)\cap\left(\Delta_{\xi<\o_1}C^{\d_{\xi}}_{\zeta}\right).
 \]
 Since countable intersection of clubs is a club, and since diagonal intersection of $\o_1$ many clubs is a club, we know that $C$ is a club in $\o_1$.
 We will prove that $C\cap \set{\a<\o_1:s_{\zeta}\conc\seq{\a}\in\widehat{S}}=\ps$.
 So suppose that an ordinal $\a<\o_1$ is such that $\a\in C\cap \set{\a<\o_1:s_{\zeta}\conc\seq{\a}\in\widehat{S}}$.
 This means that $\a\in C$ and that for some $\mu<\o_1$ and some $x\in S_{\d_{\mu}}$, $s_{\zeta}\conc\seq{\a}<x$.
 Note that this implies that
 \begin{equation}\label{assumption}
  \a\notin C^{\d_{\mu}}_{\zeta}.
 \end{equation}
 If $\mu\le\zeta$, then since $\a\in\bigcap_{\xi\le\zeta}C^{\d_{\xi}}_{\zeta}$, we have $\a\in C^{\d_{\mu}}_{\zeta}$ which is clearly contradicting (\ref{assumption}).
 Thus, it must be that $\zeta<\mu$.
 We consider two subcases, either $\a<\eta_{\mu}$ or $\a\ge\eta_{\mu}$.
 If $\a\ge\eta_{\mu}>\mu$, then $\a\in C$ implies that $\a\in\Delta_{\xi<\o_1}C^{\d_{\xi}}_{\zeta}$, which together with $a>\mu$ implies that $a\in C^{\d_{\mu}}_{\zeta}$.
 But this is in contradiction with (\ref{assumption}).
 If $\a<\eta_{\mu}$, then by the property of $\d_{\mu}$ we have $C^{\d}_{\zeta}\cap\eta_{\mu}=C^{\d_{\mu}}_{\zeta}\cap\eta_{\mu}$.
 By (\ref{assumption}), $\a\notin C^{\d_{\mu}}_{\zeta}$, so since $\a<\eta_{\mu}$, it must be that $\a\notin C^{\d}_{\zeta}$.
 But, by the definition of the set $C$, this means that $\a\notin C$, which is in contradiction with the initial assumption on $\a$.
\end{proof}

\begin{theorem}
 If $T\sst \o_1^{<\o_2}$ is an $\aleph_1$-branching $\aleph_2$-Souslin tree, then \[\o_1\times\o_2<_T D_T<_T [\o_2]^{\le\o}.\]
\end{theorem}

\begin{proof}
 Let us enumerate $T=\set{t_{\a}:\a<\o_2}$.
 
 First we prove $\o_1\times\o_2\le_T D_T$.
 We will find a Tukey function $f:\o_1\times\o_2\to D_T$.
 So, for $\seq{\b,\g}\in \o_1\times\o_2$, define
 \[
  f(\b,\g)=\set{t_{\g}}\cup\set{t_0\conc\seq{\a}\in T:\a<\b}.
 \]
 The function $f$ is properly defined because the image of every element of $\o_1\times\o_2$ is countable, thus belongs to $D_T$.
 Now we prove that $f$ is Tukey, i.e. the image of every unbounded set is unbounded.
 Let $X$ be unbounded in $\o_1\times\o_2$.
 The either $\abs{\pi_{\o_1}''X}=\aleph_1$ or $\abs{\pi_{\o_2}''X}=\aleph_2$.
 Suppose first that $\abs{\pi_{\o_2}''X}=\aleph_2$, and let $Y=\set{\g<\o_2: t_{\g}\in \pi_{\o_2}''X}$.
 Then $Y\sst \bigcup f''X$, in particular $\bigcup f''X$ is of cardiality $\aleph_2$, so $f''X$ cannot be bounded in $D_T$.
 Suppose now that $\abs{\pi_{\o_1}''X}=\aleph_1$.
 Then $\set{t_0\conc\seq{\a}:\a<\o_1}\sst\bigcup f''X$ thus showing that $f''X$ is not bounded in $D_T$.
 
 Next we prove $D_T\le_T [\o_2]^{\le\o}$.
 For $X\in D_T$ define $g(X)=\set{\a<\o_2:t_{\a}\in X}$.
 Suppose now that $\mathcal A$ is unbounded in $D_T$.
 Then $\bigcup\mathcal A$ is uncountable because otherwise $\mathcal A$ would be bounded.
 This means that
 \[
  \abs{\bigcup g''\mathcal A}=\abs{\bigcup\set{g(X):X\in \mathcal A}}=\abs{\set{\a<\o_2:t_{\a}\in\bigcup\mathcal A}}\ge\aleph_1.
 \]
 Thus, $g''\mathcal A$ is unbounded in $[\o_2]^{\le\o}$.

 Now we prove that $D_T\nleq_T \o_1\times\o_2$.
 Suppose the contrary, that there is a Tukey function $h:D_T\to \o_1\times\o_2$.
 For $\a<\o_1$, let
 \[
  D_{\a}=\set{X\in D_T: (\exists \b<\o_2)\ h(X)=\seq{\a,\b}}.
 \]
 Since $h$ is a function, $D_T=\bigcup_{\a<\o_1}D_{\a}$ is a partition.
 By Corollary \ref{part}, there is an $\a<\o_1$ and an unbounded set $E\sst D_{\a}$ of cardinality $\aleph_1$.
 Enumerate $E=\set{X_{\xi}:\xi<\o_1}$, and let $\b_{\xi}<\o_2$ be such that $h(X_{\xi})=\seq{\a,\b_{\xi}}$, and let $\b<\o_2$ be such that $\b_{\xi}<\b$ for every $\xi<\o_1$.
 Then $h''E=\set{\seq{\a,\b_{\xi}}:\xi<\o_1}$.
 Since $E$ is unbounded and $h$ is Tukey, $h''E$ is unbounded.
 This is a contradiction because $h''E\le \seq{\a,\b}$ in $\o_1\times\o_2$.
 
 Finaly, we prove $[\o_2]^{\le\o}\nleq_T D_T$.
 So suppose the contrary that $[\o_2]^{\le\o}\le_T D_T$.
 We already know that $D_T\le_T [\o_2]^{\le\o}$, so the assumption gives us $D_T\equiv_T [\o_2]^{\le\o}$.
 This means that there is a directed set $E$ such that both $D_T$ and $[\o_2]^{\le\o}$ are cofinal subsets of $E$.
 By Lemma \ref{existence} there is a set $A\sst [\o_2]^{\le\o}$ cofinal in $[\o_2]^{\le\o}$, and such that every uncountable $B\sst A$ is unbounded in $[\o_2]^{\le\o}$.
 Now, for each $x\in A$ take $d_x\in D_T$ such that $x\le d_x$.
 Since $A$ is cofinal in $E$, the set $D'=\set{d_x:x\in A}$ is also cofinal in $E$, and consequently cofinal in $D_T$.
 By Lemma \ref{fourth}, there is an uncountable subset $D^*\sst D'$ bounded in $D_T$.
 Let $d^*\in D_T$ be such that $d\le d^*$ for each $d\in D'$.
 Consider the set $B=\set{x\in A:d_x\in D'}$.
 Since $D'$ is uncountable, the set $B$ is also uncountable.
 Thus, by the assumption on $A$, the set $B$ is unbounded in $[\o_2]^{\le}$, but also in $E$ because $[\o_2]^{\le\o}$ is a cofinal subset of $E$.
 Then, for each $x\in B$ we have $x\le d_x\le d^*$, contradicting the unboundedness of $B$ in $E$.
 This shows that $D_T\not\equiv_T [\o_2]^{\le\o}$, and consequently $D_T\nleq_T \o_1\times\o_2$.
\end{proof}

This, together with the following theorem, concludes the proof of Theorem \ref{prvaizmedju}

\begin{theorem}[Gregory, see \cite{gregory}]
 If GCH holds and there is a non-reflecting stationary subset of $S^2_0$, then there is an $\aleph_1$-branching $\aleph_2$-Souslin tree.
\end{theorem}

\section{Directed set between $[\o_1]^{<\o}\times [\o_2]^{\le\o}$ and $[\o_2]^{<\o}$}

The standing assumption in this section will be that GCH holds, and that $S$ is a non-reflecting stationary subset of $S^2_0$.
GCH implies that there is a collection $\mathcal C=\set{C_{\a}:\a\in S}$ of sets of order type $\o$, such that $C_{\a}\sst \a$ for each $\a\in S$, and that for each set $X\sst \o_2$ of size $\aleph_2$ there is some $\a\in S$ such that $C_{\a}\sst X$.
Define \[D_{\mathcal C}=\set{Y\in [\o_2]^{\aleph_0}: (\forall\a\in S) \abs{Y\cap C_{\a}}<\aleph_0},\] and consider $D_{\mathcal C}$ directed by inclusion.
In this section we prove that $[\o_1]^{<\o}\times D_{\mathcal C}$ is the directed set strictly between $[\o_1]^{<\o}\times [\o_2]^{\le\o}$ and $[\o_2]^{<\o}$ in the Tukey ordering.

Since GCH implies that $\aleph_2^{\aleph_0}=\aleph_2$, we have $\abs{D_{\mathcal C}}=\aleph_2$.
Thus $\cf(D_{\mathcal C})=\aleph_2$, and consequently $D_{\mathcal C}\le_T [\o_2]^{<\o}$.
Together with $[\o_1]^{<\o}\le_T [\o_2]^{<\o}$ and Lemma \ref{upperbound}, this implies that $[\o_1]^{<\o}\times D_{\mathcal C}\le_T [\o_2]^{<\o}$. 

Since GCH implies $\abs{D_{\cc}}=\abs{[\o_2]^{\le\o}}=\aleph_2$, there is a 1-1 function $\phi:D_{\cc}\to \o_2$.
Denote $X=\set{x\cup\set{\phi(x)}:x\in D_{\cc}}$.
Clearly, $X$ is a cofinal subset of $D_{\mathcal C}$.
Since $\phi$ is 1-1, union of every uncountable subset of $X$ contains an uncountable subset of $\aleph_2$ as a subset.
Hence, every uncountable subset of $X$ is unbounded in $D_{\mathcal C}$.
This means that any 1-1 function $g:[\o_2]^{\le\o}\to X$ witnesses that $[\o_2]^{\le\o}\le_T D_{\mathcal C}$ holds.
Now $f$ mapping $[\o_1]^{<\o}\times [\o_2]^{\le\o}$ to $[\o_1]^{<\o}\times D_{\mathcal C}$, defined by $f(F,Y)=(F,g(Y))$, is Tukey. So we finally have
\[
 [\o_1]^{<\o}\times [\o_2]^{\le\o}\le_T [\o_1]^{<\o}\times D_{\mathcal C}\le_T [\o_2]^{<\o}
\]
In the remainder of this section we show that these inequalities are strict.

\begin{lemma}
 There is no Tukey map from $D_{\mathcal C}\times [\o_1]^{<\o}$ into $[\o_2]^{\le\o}\times [\o_1]^{<\o}$.
\end{lemma}

\begin{proof}
 Suppose that $f:D_{\mathcal C}\times [\o_1]^{<\o}\to [\o_2]^{\le\o}\times [\o_1]^{<\o}$ is Tukey.
 For each $\xi<\o_2$ denote $(x_{\xi},s_{\xi})=f(\set{\xi},\ps)$.
 Consider the set $\set{(x_{\xi},s_{\xi}):\xi<\o_2}$.
 Since $[\o_1]^{<\o}$ is of size $\aleph_1$, there is some $X\sst \o_2$ of size $\aleph_2$, and $s\in [\o_1]^{<\o}$ such that $s_{\xi}=s$ for each $\xi\in X$.
 By the assumption on $\mathcal C$, there is some $\a\in S$ such that $C_{\a}\sst X$.
 Now, the set $A=\set{(\set{\xi},\ps):\xi\in C_{\a}}$ is unbounded in $D_{\mathcal C}\times [\o_1]^{<\o}$.
 In particular the set $\pi_{D_{\mathcal C}}''A$ is unbounded in $D_{\mathcal C}$, as witnessed by the infinite intersection $\pi_{D_{\mathcal C}}''A\cap C_{\a}$.
 But the set $f''A=\set{(x_{\xi},s):\xi\in C_{\a}}$ is bounded in $[\o_2]^{\le\o}\times [\o_1]^{<\o}$ contradicting the assumption that $f$ is a Tukey map.
 To see that $f''A$ is bounded, note that $C_{\a}$ is countable so the bound for $f''A$ is $(\bigcup_{\xi\in  C_{\alpha}}x_{\xi},s)$. This is because $\bigcup_{\xi\in C_{\a}}x_{\xi}$ is in $[\o_2]^{\le\o}$ (being a countable union of countable sets).
\end{proof}

\begin{lemma}
 There is no Tukey map from $[\o_2]^{<\aleph_0}$ into $[\o_1]^{<\o}\times D_{\mathcal C}$.
\end{lemma}

\begin{proof}
 Suppose that the lemma fails.
 This means that $[\o_2]^{<\aleph_0}\equiv_T [\o_1]^{<\o}\times D_{\mathcal C}$, i.e. that there is $Y\sst [\o_1]^{<\o}\times D_{\mathcal C}$ of size $\aleph_2$ such that every infinite subset of $Y$ is unbounded in $D_{\mathcal C}$.
 Since $Y$ is of size $\aleph_2$, there is a finite $F\sst \o_1$, and $X\sst D_{\mathcal C}$ of size $\aleph_2$ such that $(F,x)\in Y$ for each $x\in X$.
 Let $X=\set{x_{\a}:\a<\o_2}$.
 Clearly, every infinite subset of $X$ is unbounded in $D_{\mathcal C}$.
 
 Using GCH we may assume that $X$ forms an increasing $\Delta$-system with the root $X_0$.
 In other words, for $\a<\b<\o_2$ we have $x_{\a}\cap x_{\b}=X_0$ and $x_{\a}\setminus X_0<x_{\b}\setminus X_0$.
 Take some large enough $\theta$ so that all relevant object belong to $H(\theta)$.
 Pick a continuous increasing sequence $\seq{M_{\xi}:\xi<\o_2}$ of elementary submodels of $H(\theta)$ of cardinality $\aleph_1$, such that $X_0\in M_0$.
 Now the set $C=\set{\delta<\o_2:M_{\d}\cap\o_2=\d}$ is a club in $\o_2$.
 Let $\g\in C$ be such that $\cf(\g)=\o_1$ and $\sup(C\cap \g)=\g$.
 Such a $\g$ exists because $C$ is a club.
 Since $S$ is non-reflecting, there is $E\sst C\cap \g$, club in $C\cap \g$ such that $\otp(E)=\o_1$ and $E\cap S=\ps$.
 Let $E=\set{\d_{\a}:\a<\o_1}$ be an increasing enumeration of $E$.
 For each $\a<\o_1$ let $\d_{\a}^+=\min(E\setminus \d_{\a})$.
 By elementarity, there is a sequence $\seq{x_{\xi_{\a}}:\a<\o_1}$ such that $x_{\xi_{\a}}\setminus X_0\sst M_{\d^+_{\a}}\setminus M_{\d_{\a}}$.
 Consider the set $x=\bigcup_{n<\o}x_{\xi_n}$ and take $\a\in S$.
 Then $\a\notin E$, so either $\a>\sup_{n<\o}\d_n$ or $\a<\sup_{n<\o}\d_n$.
 If $\a>\sup_{n<\o}\d_n$, then $C_{\a}\cap x$ is a finite set.
 If $\a<\sup_{n<\o}\d_n$, then there is $n_0$ so that $\d_{n_0+1}>\a$.
 Since $C_{\a}\cap x_{\xi_m}$ is finite for each $m\le n_0$, and $C_{\a}\cap x=\bigcup_{m\le n_0}(C_{\a}\cap x_{\xi_m})$ it must be that $C_{\a}\cap x$ is finite in this case also.
 Thus $x\in D_{\mathcal C}$ is a bound for a countable set $\set{x_{\xi_n}:n<\o}\sst X$, which is in contradiction with the choice of $X$. So the lemma is proved.
\end{proof}

This concludes the proof of Theorem \ref{drugaizmedju}.

\end{document}